\documentclass[11pt]{amsart}
\usepackage{graphicx}
\usepackage{amsmath, amsthm, amssymb}
\usepackage{url}

\theoremstyle{definition}
\newtheorem{thm}{Theorem}[section]
\newtheorem{lem}[thm]{Lemma}

\newtheorem*{defin}{Definition}
\newtheorem*{rem}{Remark}

\begin{document}

\title[Poincar\'e Series of Non-divisorial valuations]{Poincar\'e series of non-divisorial valuations on two-dimensional function fields}

\author{Charles Li}
\author{Hans Schoutens}

\address{Department of Mathematics and Computer Sciences, Mercy College, 555 Broadway, Dobbs Ferry, NY 10522, USA}
\email{cli2@mercy.edu}

\address{Department of Mathematics, 365 5th Avenue, The CUNY Graduate Center, New York, NY 10016, USA}
\email{hschoutens@citytech.cuny.edu}

\date{}

\maketitle

\begin{abstract}
We use dual graphs and generating sequences of valuations to compute the Poincar\'e series of non-divisorial valuations on function fields of dimension two. The Poincar\'e series are shown to reflect data from the dual graphs and hence carry equivalent information. This extends Galindo's earlier results on the Poincar\'e series of divisorial valuations and consequently offers an alternative classification of valuations via Poincar\'e series.
\end{abstract}

\section{Introduction}

This paper is a follow-up to the results in \cite{li}. The analysis of dual graphs and generating sequences of valuations of function fields of dimension two in the previous paper leads to explicit computations for the Poincar\'e series of non-divisorial valuations. Taken together with the Poincar\'e series computations for divisorial valuations in \cite{gal}, Poincar\'e series thus offer an alternative classification of valuations, the main purpose of this paper.

This paper resulted from investigating the ``average of a valuation,'' a concept introduced in unpublished notes by the second author, who was the first author's thesis advisor. The average of a valuation turned out to be equivalent to the Poincar\'e series of a valuation, but the investigation led down a fertile path, ultimately bearing fruit in this paper and its predecessor.

Valuations on function fields of dimension two have been studied via their generating sequences, dual graphs and Poincar\'e series. It is believed by many who work with valuations that these perspectives are equivalent to some degree. This paper and its predecessor aims to develop such connections between these various ways of working with valuations, and to extend the classification of valuations on function fields of dimension two.

The value semigroup was first used to define the Poincar\'e series for a divisorial valuation centered on a local domain in \cite{gal}. Given Hironaka's landmark work on the resolution of singularities after a finite sequence of blowups, it made sense to look at the Poincar\'e series of divisorial valuations and the multi-graded Poincar\'e series of a finite set of divisorial valuations. Some papers in this direction for surfaces are \cite{cg}, \cite{chr} and \cite{dgz}. 

We are interested in computing the Poincar\'e series for the non-divisorial cases. In particular, after Zariski's work on the classification of valuations, it is well-known that valuations on function fields of surfaces can be sorted into five types using Abhyankar's inequality. Divisorial valuations are the most important and widely used of the five cases. However, we wish to fill a void in the literature by dealing with the Poincar\'e series in the remaining four cases. Poincar\'e series would then give an alternative classification of valuations.

\section{classification of valuations} \label{sectcv}

Abhyankar's inequality can be used  to classify the valuations on function fields of dimension 2 according to the classical invariants: rank, rational rank $rr$ and dimension $d$, i.e. the transcendence degree $d$ of the residue field over the base field. In this setting, Abhyankar's inequality states $rr+d\leq 2$.

\begin{thm} (Classification of Valuations)
\\
Let $(R,\mathfrak{m})$ be a 2-dimensional regular local ring whose fraction field $K$ is a function field of dimension 2 over an algebraically closed field $k$ of characteristic 0. Let $\nu$ be a valuation on $K$ centered on $R$. Let $(V,\mathfrak{m}_V)$ be the valuation ring and its maximal ideal,  hence $\text{Frac}(V)=K=\text{Frac}(R)$. Then $\nu$ is one of the following cases:
\begin{equation*}
\begin{array}{|c|c|c|c|c|c|} \hline
\text{Rank} & \text{rr} & d & \text{Discreteness} & \text{Value group} & \text{Type} \\ \hline
1 & 1 & 1 & \text{discrete} & \mathbb{Z} & 0 \\ \hline
1 & 1 & 0 & \text{non-discrete} & \text{additive subgroup of}\left.\right. \mathbb{Q}& 1 \\ \hline
1 & 2 & 0 & \text{non-discrete} & \mathbb{Z} + \mathbb{Z} \tau, \left.\right.\text{where} \left. \tau \right. \text{is irrational} & 2 \\ \hline 
2 & 2 & 0 & \text{discrete} & \mathbb{Z}^{2} & 3 \left.\right.\text{and} \left.\right. 4.2\\ \hline
1 & 1 & 0 & \text{discrete} & \mathbb{Z} & 4.1 \\ \hline
\end{array}
\end{equation*}
where discreteness refers to the discrete or non-discrete nature of the value groups, and where the value groups are given up to order isomorphism.
\end{thm}

More details on the classification with respect to the classical valuation theoretic invariants can be found in \cite{cut}.

The types originate from Spivakovsky's work classifying valuations according to their dual graphs. One difference in notation: we denote Types 4.1 and 4.2 to reflect the rank of the valuation. These two types were originally switched in \cite{spiv}. More details on dual graphs and generating sequences can be found in \cite{li}.

\section{generating sequences} \label{sectgs}

Generating sequences can be used to gain a handle on the value semigroup. In this section, we briefly summarize some of the main points and establish notation.

\begin{defin}
The {\it value semigroup} is:
\begin{equation*}
S:=\left\{\nu(x) \left.|\right. x\in \mathfrak{m} \right\}
\end{equation*}
\end{defin}

\begin{defin}
Let $\{Q_i\}_{i=0}^{g'}$ be a (possibly infinite) sequence of elements of $\mathfrak m$. We say that $\{Q_i\}$ is a {\it generating sequence} for $\nu$ if

\begin{equation*}
S=\left\{\sum_{i} \alpha_i \nu(Q_i) \left|\right. \alpha_i \in \mathbb{N}_0\right\}
\end{equation*}
A {\it minimal generating sequence} is one in which exclusion of any $Q_j$ will cause $\{Q_i\}_{i\neq j}$ to not be a generating sequence. Note that $\mathbb{N}_0$ denotes the set of non-negative integers.
\end{defin}

\begin{rem}
This definition is suitable for our purposes, but is different from the usual definition of generating sequences involving generating the value ideals instead of generating the value semigroup. Using this definition leads to a significant difference in the Type 4.1 case. See \cite{li} for further details.
\end{rem}

The dual graph of a valuation has $g$ dual graph pieces with 2 or more segments each. If there is a tail piece with just one segment at the end of the dual graph, then it is considered the $(g+1)$-th piece. Setting $g'$ as in the following table gives the correct number of elements in a minimal generating sequence for each of the types.

\begin{equation*}
\begin{array}{|c|c|} \hline
\text{Type} & g'\\ \hline
1 & \infty \\ \hline
2 & g \\ \hline
3 & g \\ \hline
4.1 & g \\ \hline
4.2 & g+1 \\ \hline
\end{array}
\end{equation*}

\begin{thm} \label{unique} (Unique representation)
\\
Let $\nu$ be a non-divisorial valuation. Let $\beta_i=\nu(Q_i)$. Let $s\in S$. In the $i$-th dual graph piece, let $a_j^{(i)}$ be the number of vertices in the $j$-th segment, and let $m_i$ be the number of segments. Let the continued fraction $\displaystyle \left[a_1^{(i)},a_2^{(i)},\ldots,a_{m_i}^{(i)},1\right]$ simplify as $\displaystyle \frac{p_i}{q_i}$. Assume $\nu$ is not Type 4.1. We may uniquely write:
\begin{equation*}
\begin{array}{rll}
\displaystyle s=\sum_{i=0}^{g'} \alpha_i \beta_i, & \text{ where } \alpha_0\in\mathbb{N}_0, \alpha_{g'}\in\mathbb{N}_0 \text{ and } 0 \leq \alpha_i \leq q_i-1 \text{ for } 1\leq i < g'
\end{array}
\end{equation*}
If $\nu$ is Type 4.1, then we may uniquely write:
\begin{equation*}
\begin{array}{rll}
\displaystyle s=\sum_{i=0}^{g} \alpha_i \beta_i, & \text{where}\left.\right. \alpha_0\in\mathbb{N}_0, \left.\right.\text{and}\left.\right. 0 \leq \alpha_i \leq q_i-1 \left.\right.\text{for}\left.\right. 1\leq i \leq g
\end{array}
\end{equation*}
\end{thm}
\begin{proof}
See \cite{li}. The Type 1 case of this theorem was proved in \cite{ghk}.
\end{proof}

\section{poincar\'e series} \label{sectps}

A Poincar\'e series can be thought of as a valuation theoretic analogue to Hilbert series which encodes the lengths of values in the value semigroup.

\begin{defin}
Let $I_s:=\{x\in R \left.|\right. \nu(x)\geq s\}$ and $I_s^+:=\{x\in R \left.|\right. \nu(x)> s\}$, where $s\in S$. The {\it associated graded algebra} (over $k$) is: 
\begin{equation*}
\text{gr}_{\nu}(R):=\bigoplus_{s\in S} \frac{I_s}{I_s^+}
\end{equation*}
\end{defin}

Geometric interpretations of the associated graded algebra can be found in \cite{tei}. 

\begin{defin}
The {\it length} of $s$, denoted $l(s)$, is the length of $I_s/I_s^+$ considered as a module over $R/\mathfrak{m}$, i.e. as a $k$-module.
\end{defin}

\begin{defin}
The {\it Poincar\'e series} of $\nu$ is the formal sum in $\displaystyle \mathbb{Z}[ [t^{\Gamma}] ]$:
\begin{equation*}
{\mathcal P}_{\nu}(t)=\sum_{s\in S}l(s)t^{s}
\end{equation*}
\end{defin}

Galindo studied the Poincar\'e series of divisorial valuations and established the following theorem. We will extend this theorem to the non-divisorial valuations in the rest of this paper.

\begin{thm} (Galindo)

In the divisorial case, i.e. $\nu$ has a finite dual graph, the Poincar\'e series is:
\begin{equation*}
\left\{
\begin{array}{ll}
\displaystyle \mathcal P_{\nu}(t)=\frac{1}{1-t^{\beta_0}}\cdot\prod_{j=1}^g\frac{1-t^{q_j\beta_j}}{1-t^{\beta_j}}\cdot\frac{1}{1-t^{\beta_{g+1}}} & \text{if } a_1^{(g+1)} \neq 0,
\\\\
\displaystyle \mathcal P_{\nu}(t)=\frac{1}{1-t^{\beta_0}}\cdot\prod_{j=1}^{g-1}\frac{1-t^{q_j\beta_j}}{1-t^{\beta_j}} \cdot\frac{1}{1-t^{\beta_{g}}} & \text{if } a_1^{(g+1)}=0
\end{array}
\right.
\end{equation*}
\end{thm}
\begin{proof}
See Theorem 1 in \cite{gal}.
\end{proof}

At this point, we need one more lemma to compute the Poincar\'e series for the remaining non-divisorial cases.

\begin{lem} \label{length}
Let $\nu$ be a non-divisorial valuation. The lengths $l(s)=1$ for all $s\in S$.
\end{lem}

\begin{proof}
Assume $l(s)>1$ for some $s\in S$, hence $\text{dim}_k(I_s/I_s^+)>1$. We may pick $r_1, r_2\in R$ that are representatives of two different equivalence classes in $I_s/I_s^+$, i.e. such that $\nu(r_1)=s=\nu(r_2)$ and $r_1-r_2\notin I_s^+$. If $r_1/r_2\in k$, then $r_1/r_2\cdot r_2=r_1$ and $r_1\sim r_2$. Hence $\displaystyle r_1 \not \sim r_2$ implies $r_1/r_2 \notin k$. On the other hand, $\nu(r_1/r_2)=0$ so $r_1/r_2\in V/\mathfrak{m}_V$. The residual transcendence degree is 0 in the non-divisorial cases so $V/\mathfrak{m}_V$ is an algebraic extension over $R/\mathfrak{m}\cong k$, hence $V/\mathfrak{m}_V\cong k$ since $k$ is algebraically closed. Thus, we get $r_1/r_2\in k$, a contradiction.
\end{proof}

\begin{thm} \label{poincare}
In the non-divisorial cases, the Poincar\'e series are:

\begin{equation*}
\begin{array}{ll}
\text{Type 1:} &
\displaystyle \mathcal P_{\nu}(t)=\frac{1}{1-t^{\beta_0}}\cdot\prod_{j=1}^{\infty}\frac{1-t^{q_j\beta_j}}{1-t^{\beta_j}}
\\\\
\text{Type 2:} &
\displaystyle \mathcal P_{\nu}(t)=\frac{1}{1-t^{\beta_0}}\cdot\prod_{j=1}^{g-1}\frac{1-t^{q_j\beta_j}}{1-t^{\beta_j}}\cdot\frac{1}{1-t^{\beta_{g}}}
\\\\
\text{Type 3:} &
\displaystyle \mathcal P_{\nu}(t)=\frac{1}{1-t^{\beta_0}}\cdot\prod_{j=1}^{g-1}\frac{1-t^{q_j\beta_j}}{1-t^{\beta_j}}\cdot\frac{1}{1-t^{\beta_{g}}}
\\\\
\text{Type 4.1:} &
\displaystyle \mathcal P_{\nu}(t)=\frac{1}{1-t^{\beta_0}}\cdot\prod_{j=1}^g\frac{1-t^{q_j\beta_j}}{1-t^{\beta_j}}
\end{array}
\end{equation*}
\begin{equation*}
\begin{array}{ll}
\text{Type 4.2:} &
\displaystyle \mathcal P_{\nu}(t)=\frac{1}{1-t^{\beta_0}}\cdot\prod_{j=1}^g\frac{1-t^{q_j\beta_j}}{1-t^{\beta_j}}\cdot\frac{1}{1-t^{\beta_{g+1}}}
\end{array}
\end{equation*}
\end{thm}

\begin{proof}
The lengths are all 1 by Lemma \ref{length}. We may uniquely represent the values in $S$ in terms of $\{\beta_i\}$ by Theorem \ref{unique}. For Type 1 valuations:
\begin{equation*}
\begin{array}{rl}
\displaystyle \mathcal{P}_{\nu}(t) & \displaystyle =\sum_{s\in S}l(s)t^{s}
\\\\
& \displaystyle =\sum_{\alpha_0,\ldots,\alpha_{i},\ldots} t^{\sum_i \alpha_i \beta_i}
\\\\
& \displaystyle =\left(\sum_{\alpha_0\in\mathbb{N}_0} t^{\alpha_0\beta_0}\right) \left(\sum_{\alpha_1=0}^{q_1-1} t^{\alpha_1\beta_1}\right) \cdots \left(\sum_{\alpha_i=0}^{q_i-1} t^{\alpha_i\beta_i}\right) \cdots
\\\\
& \displaystyle =\frac{1}{1-t^{\beta_0}}\cdot\prod_{j=1}^{\infty}\frac{1-t^{q_j\beta_j}}{1-t^{\beta_j}}
\end{array}
\end{equation*}
We tackle Types 2, 3 and 4.2 simultaneously:
\begin{equation*}
\begin{array}{l}
\displaystyle \mathcal{P}_{\nu}(t) = \displaystyle \sum_{s\in S}l(s)t^{s}
\\\\
\displaystyle =\sum_{\alpha_0,\ldots,\alpha_{g'}} t^{\sum_i \alpha_i \beta_i}
\\\\
\displaystyle =\left(\sum_{\alpha_0\in\mathbb{N}_0} t^{\alpha_0\beta_0}\right) \left(\sum_{\alpha_1=0}^{q_1-1} t^{\alpha_1\beta_1}\right) \cdots \left(\sum_{\alpha_{g'-1}=0}^{q_{g'-1}-1} t^{\alpha_{g'-1}\beta_{g'-1}}\right) \left(\sum_{\alpha_{g'}\in\mathbb{N}_0} t^{\alpha_{g'}\beta_{g'}}\right)
\end{array}
\end{equation*}
\begin{equation*}
\begin{array}{l}
\displaystyle =\frac{1}{1-t^{\beta_0}}\cdot\prod_{j=1}^{g'-1}\frac{1-t^{q_j\beta_j}}{1-t^{\beta_j}}\cdot\frac{1}{1-t^{\beta_{g'}}}
\end{array}
\end{equation*}
Lastly, for Type 4.1 valuations: 
\begin{equation*}
\begin{array}{rl}
\displaystyle \mathcal{P}_{\nu}(t) & \displaystyle =\sum_{s\in S}l(s)t^{s}
\\\\
& \displaystyle =\sum_{\alpha_0,\ldots,\alpha_{g}} t^{\sum_i \alpha_i \beta_i}
\\\\
& \displaystyle =\left(\sum_{\alpha_0\in\mathbb{N}_0} t^{\alpha_0\beta_0}\right) \left(\sum_{\alpha_1=0}^{q_1-1} t^{\alpha_1\beta_1}\right) \cdots \left(\sum_{\alpha_{g}=0}^{q_{g}-1} t^{\alpha_{g}\beta_{g}}\right)
\\\\
& \displaystyle =\frac{1}{1-t^{\beta_0}}\cdot\prod_{j=1}^{g}\frac{1-t^{q_j\beta_j}}{1-t^{\beta_j}}
\end{array}
\end{equation*}
\end{proof}

\end{document}